\documentclass[final,1p,times]{elsarticle}
\usepackage{amssymb}
\usepackage{amsmath}
\usepackage{amsthm}
\usepackage[utf8]{inputenc}
\usepackage{marginnote}
\usepackage{amsfonts}
\usepackage{textcomp}
\usepackage{color,soul}
\usepackage{placeins}
\usepackage{tikz}
\usepackage{makecell}
\usepackage{amssymb}
\usepackage{graphicx}
\usepackage{hyperref}
\usepackage[none]{hyphenat}
\usepackage{lipsum}
\newtheorem{thm}{Theorem}[section]
\newtheorem{lem}[thm]{Lemma}
\newtheorem{rem}[thm]{Remark}

\newtheorem{cor}[thm]{Corollary}
\newtheorem{example}[thm]{Example}
\newdefinition{defn}[thm]{Definition}
\allowdisplaybreaks
\numberwithin{equation}{section}
\begin{document}
	\begin{frontmatter}
		\title{Optimal dual pairs of frames for erasures}
		\author{S. Arati }
		\ead{aratishashi@iisertvm.ac.in}
		\author{P. Devaraj\corref{deva}}
		\ead{devarajp@iisertvm.ac.in}
		\author{Shankhadeep Mondal}
		\ead{shankhadeep16@iisertvm.ac.in}

		\address{School of Mathematics, Indian Institute of Science Education and Research Thiruvananthapuram, Maruthamala P.O, Vithura, Thiruvananthapuram-695551.}
		\cortext[deva]{Corresponding Author:devarajp@iisertvm.ac.in}

\begin{abstract}
The study involves characterizations of dual pairs of frames which are optimal to handle erasures among all dual pairs for a finite dimensional Hilbert space. A new optimality measure using  the Frobenius norm of the error operator has been introduced and the corresponding optimal dual pairs have been analyzed for any number of erasures. Also, other measures of the error operator, namely the spectral radius and the numerical radius, have been considered for the analysis. Besides, explicit construction of certain optimal dual pairs has been provided.
\end{abstract}


\end{frontmatter}
\section{Introduction}

Duffin and Schaeffer \cite{duff} introduced the notion of frames which generalizes  the concept of bases in a Hilbert space. A vector in a Hilbert space can have several representations using a sequence of frame elements.
Erasures of frame coefficients may occur when transmitting data, especially under challenging global network conditions like congestion and limitations in the transmission channel. The redundancy features of frames plays a crucial role in easing the effect of these erasures and improving the precision of reconstructing the original signal. Recent years have witnessed substantial research endeavors aimed at tackling erasure-related challenges through the application of frame theory.
\par
 Goyal et al. \cite{goya} have analyzed the optimality of unit norm tight frames for handling erasures in a coding theory perspective.  They have considered the mean-square error as a measure of optimality and  shown that for one erasure, a unit norm tight frame is optimal among all unit norm frames,  both in average and worst-case scenarios. Casazza and Kova{\v{c}}evi{\'c} \cite{casa2} have extensively investigated equal norm tight frames and analyzed their construction and robustness against erasures.  In a related study, Holmes and Paulsen \cite{holm} have introduced the concept of utilizing the operator norm of the error operator as an optimality measure and have classified those dual pairs which minimize the worst case error among all dual pairs of the form $(F,F)$ consisting of  Parseval frames and their canonical duals. A similar problem but with an average type optimality measure, considering various possible locations of erasures of a particular length, has been dealt with in \cite{bodm1}. In the context of a given frame, Lopez and Han \cite{jerr} have provided a sufficient condition for the canonical dual to be the only optimal dual with respect to the operator norm.  Furthermore, their study involves the analysis of  various topological properties of the collection of all optimal duals associated with a given frame. Moreover, in the context as described above, additional characterizations for optimality have been provided in \cite{jins}. Pehlivan et al. \cite{sali} have taken the spectral radius of the error operator as the measure of optimality and have given  characterizations of optimal dual frames  of  a given frame for one erasure. The  same problem associated with the occurrence  of  two erasures has  been  looked into  in \cite{peh, dev}. Subsequently, many other optimality measures have been considered for the analysis of optimal dual frames. For instance, numerical radius has been introduced as a measure in \cite{arab}, the average of numerical radius and operator norm of the error operator has been considered in \cite{deep} and seminorms of the error operator have been used in \cite{key}. Interestingly, the erasure problem has also been formulated by incorporating probabilities and weights associated with the erasures. We refer to \cite{leng2, leng3, li2, li, adm} in this connection.
\par
In the works discussed above, most of them deal with finding an optimal dual frame for a given fixed frame, under various measures of optimality. Some of them consider the optimal dual pair among pairs of Parseval frames and their canonical duals. In \cite{adm}, the optimality has been analyzed, while taking into account the larger collection of all possible pairs of frames and their duals. However, their study is confined to the cases of spectral radius and operator norm as the optimality measure. In this paper, we further extend the study of the deterministic erasure problem considering all dual pairs to other contexts, such as the Frobenius norm, which have not been studied so far as per our knowledge. More precisely, we give many characterizations of dual pairs which, among all pairs of frames and their duals, minimize the maximum of the Frobenius norm of the error operators associated with all possible locations of erasures of any arbitrary but fixed length. The characterizing conditions pertaining to the particular case of optimality of tight frames and their canonical duals, among all pairs, involve the well-known concept of equiangular frames, which is prominent in quantum information theory. Furthermore, we provide some results related to the spectral radius measure, which turn out to be an improvement of the corresponding results that can be obtained from \cite{adm}. In addition, we have made attempts to construct spectrally optimal dual pairs for both one and two erasures. Besides, we analyse the scenario wherein the numerical radius of the error operator measures the optimality, in the case of one erasure. These problems, associated with the three different measures, have been addressed in Sections 3, 4 and 5 respectively. A comparitive analysis between the three optimal duals is provided in Section 6. We shall now turn to Section 2 to recall the necessary background in connection with the problem of erasures and optimality in reconstruction.

\section{Preliminaries on erasures}\label{section2}
We shall denote an $n-$dimensional Hilbert space by $\mathcal{H}$. A finite \textit{frame} for $\mathcal{H}$ is a finite sequence $F= \{f_i\}_{i=1}^N$ in $\mathcal{H}$ such that
 $$\displaystyle{A \left\| f \right\|^2\leq \sum\limits_{i=1}^N \left| \langle f, f_i \rangle \right|^2 \leq B \left\| f \right\|^2}, \forall f\in \mathcal{H},$$
where $A$ and $B$ are  some positive constants and are called lower and upper frame bounds respectively. The supremum over all such lower bounds and the infimum over all such upper bounds are called the optimal frame bounds. If there exists a constant $\alpha$ such that  $\left\| f_i \right\|=\alpha $ for all $i$,
 then the frame  $F$ is called an \textit{equal norm frame}. An equal norm frame $F$ is said to be an \textit{equiangular frame} if there exists a constant $c$ such that $|\langle f_{i},f_{j}\rangle| =c $ for all $i\neq j.$
 A frame $F$ for $\mathcal{H}$ is called a \textit{tight frame} if there exists a positive constant  $A$ such that  $\sum\limits_{i=1}^N\big| \langle f,f_i \rangle \big|^2 = A\left\| f \right\|^2$ for all $f\in \mathcal{H},$  and is called a \textit{Parseval frame} if $\sum\limits_{i=1}^N\big| \langle f,f_i \rangle \big|^2 = \left\| f \right\|^2$ for all $f\in \mathcal{H}.$ There are three important operators, associated with each frame $F$ for $\mathcal{H}$:
\begin{itemize}
\item \textit{Analysis operator} $\Theta_F: \mathcal{H} \to {\mathbb{C}^N} $ given by $\Theta_F(f)= \{\langle f,f_i \rangle\}_{i=1}^N$,
\item Adjoint of $\Theta_F$, called \textit{synthesis operator} $\Theta_{F}^*: \mathbb{C}^N \to \mathcal{H} $ given by $\Theta_{F}^*\left(\{c_i\}_{i=1}^N\right) = \sum\limits_{i=1}^N {c_i f_i} .$
\item \textit{Frame operator} $S_F : \mathcal{H}\to \mathcal{H} $  defined by $S_{F}=\Theta_{F}^{*}\Theta_{F}.$
\end{itemize}
The explicit action of the frame operator is $ S_{F} f= \sum_{i=1}^N \langle f,f_i \rangle f_i$, $f\in\mathcal{H}$. It is an invertible, self-adjoint and positive operator on $\mathcal{H}.$ It is well-known that $\{S_{F}^{-1}f_i\}_{i=1}^{N}$ is  a frame for  $\mathcal{H}$ and is known as the canonical or standard dual of $F$. It provides the following reconstruction formula: $f= \sum\limits_{i=1}^N \langle f,S_{F}^{-1}f_i\rangle f_i,$ for all $f\in \mathcal{H}.$ In general, for a frame $F$ for $\mathcal{H},$ a sequence $G=\{g_i\}_{i=1}^N$ in $\mathcal{H}$ satisfying $f = \sum\limits_{i=1}^N\langle f,g_i \rangle f_i $ for all $f \in \mathcal{H},$ is called a dual frame of $F.$ In fact, computing the adjoint of the operator
$f\mapsto \sum\limits_{i=1}^N\langle f,g_i \rangle f_i $ suggests that $f = \sum\limits_{i=1}^N\langle f,f_i \rangle g_i $ for all $f \in \mathcal{H}.$  This implies that  $G$ spans $\mathcal{H},$ and hence
$G$ is also a frame for $\mathcal{H}$.
A pair $(F,G)$ comprising of a frame and its dual frame, having $N$ elements each, is said to be an $(N,n)$ dual pair for $\mathcal{H}.$  Every dual frame $G$ of a frame $F$ can be  expressed in terms of the canonical dual of $F$ as $G=\{S_F^{-1}f_i +u_i\}_{i=1}^N,$ where $\{u_i\}_{i=1}^N$ is a sequence in $\mathcal{H}$ satisfying $\sum_{i=1}^N \langle f,f_i \rangle u_i = \sum_{i=1}^N \langle f,u_i \rangle f_i = 0,$ for every $f\in \mathcal{H}.$ Moreover, for an $(N,n)$ dual pair, we have
\begin{equation} \label{eqn2point1}
\sum_{i=1}^N\langle g_i,f_i \rangle= tr(\Theta_F \Theta_{G}^*)=tr (\Theta_{G}^* \Theta_F)= tr(I)=n.
\end{equation}
For more details on frame theory, we refer to \cite{heil,CasaKuty,ole}.
\par
The error operator associated with the transmission loss of $m$, $1\leq m\leq N$, frame coefficients in the frame expansion involving an $(N,n)$ dual pair $(F,G)$ is defined by
$$E_{\Lambda,F,G}f:= \Theta_{G}^*D\Theta_{F} f=\sum_{i\in\Lambda} \langle f,f_i \rangle g_i,\quad f\in\mathcal{H},$$ where $\Lambda $ denotes the set of $m$ indices at which the coefficients are lost, $D=(d_{ij})$ is a diagonal matrix of order $N$ and  having entries $d_{ii}= 1$ for $i\in \Lambda$ and 0 otherwise. Considering all possible locations for the $m$ erasures and taking $${\cal A}_{m}=\left\{ \Lambda=\{i_1,i_2,\ldots,i_m\} \subseteq \{1,2, \ldots, N \}:i_1<i_2<\cdots <i_m\right\},$$
the maximum  error is
 $$ \max \bigg\{ \mathcal{M}( E_{\Lambda,F,G} ) : \Lambda \in {\cal A}_{m}  \bigg\}, $$
 where $\mathcal{M}$ is a fixed measure of the error operator. The objective is to find that dual pair $(F,G)$ which minimizes this maximum error among all $(N,n)$ dual pairs. In our study, we take three different choices of $\mathcal{M}$, namely the Frobenius norm, the spectral radius and the numerical radius. We require, for our analysis, the notions of $1-$uniform and $2-$uniform dual pairs, defined as follows.
\begin{defn}
A dual pair $(F,G)$ in a Hilbert space $\mathcal{H}$ is said to be  $1-$ uniform if there exists a constant $c'$ such that  $\langle f_i ,g_i \rangle =c' $  for all $i$. Moreover, this constant turns out to be $\frac{n}{N}.$
 \end{defn}

 \begin{defn}
 A dual pair $(F,G)$ in a Hilbert space $\mathcal{H}$ is said to be $2-$uniform if it is $1-$uniform and there exists a constant $c''$ such that   $\langle f_i ,g_j \rangle \langle f_j ,g_i \rangle=c''$  for $i \neq j$.
 \end{defn}

 \section{ Optimality analysis with respect to the Frobenius norm}
In this section, we analyze optimal dual pairs for erasures, using the Frobenius norm of the error operator as the optimality measure. Let $1 \leq m \leq N.$ For an $(N,n)$ dual pair $(F,G)$ for $\mathcal{H},$  we define $ \epsilon_{F,G}^{(m)} := \max \bigg\{ \left\| E_{{\Lambda,F,G}} \right\|_\mathcal{F} : \Lambda\in\mathcal{A}_m \bigg\},$ where  the Frobenius norm of an operator $T$ is given by $\|T\|_{\mathcal{F}} =\sqrt{tr(T^* T)}$ with $tr$ denoting the trace. Further,
\begin{align*}
& \epsilon^{(1)} := \inf \bigg\{ \epsilon_{F,G}^{(1)} : (F,G)\; \text{is an}\; (N,n)\; \text{dual pair for}\; \mathcal{H} \bigg\}, \\&{\mathcal{F}}^{(1)} := \left\{ (F',G') : \epsilon_{F',G'}^{(1)} = \epsilon^{(1)} \right\},\\&\epsilon^{(m)} := \inf \left\{ \epsilon_{F,G}^{(m)} : (F,G) \in {\mathcal{F}}^{(m-1)} \right\},\quad m>1,\\& {\mathcal{F}}^{(m)} := \left\{ (F',G')  \in {\mathcal{F}}^{(m-1)} : \epsilon_{F',G'}^{(m)} = \epsilon^{(m)} \right\}, \quad m>1.
\end{align*}
The members of ${\mathcal{F}}^{(m)}$ are called $m-$erasure optimal dual pairs under Frobenius norm. It is an easy observation that ${\mathcal{F}}^{(N)}={\mathcal{F}}^{(N-1)}.$ Our main goal is to characterize optimal dual pairs when there are $m$ erasures, where $m=1,2,\ldots,N$. The following theorem pertains to the above problem for one erasure.
\begin{thm} \label{prop3point2}
 The following hold :
\begin{enumerate}
	\item [{\em (i)}]  $\epsilon^{(1)} = \frac{n}{N}$ and ${\mathcal{F}}^{(1)} = \left\{(F,G) : \|f_i\|\,\|g_i\| = \frac{n}{N},\; 1 \leq i \leq N  \right\}.$
	\item [{\em (ii)}] If $F$ is a tight frame, then $(F, S_{F}^{-1}F) \in {\mathcal{F}}^{(1)} $ if and only if $F$ is an equal norm frame.
\end{enumerate}
\end{thm}

\begin{proof}
 (i) For any $(N,n)$ dual pair $(F,G)$\, for $\mathcal{H},$  we have  $n= \sum\limits_{i=1}^N \langle f_i , g_i \rangle \leq \sum\limits_{i=1}^N \|f_i\|\,\|g_i\|,$ by using \eqref{eqn2point1}. Subsequently, we get $\max\limits_{1 \leq i \leq N} \|f_i\|\,\|g_i\| \geq \frac{n}{N}. $ If the erasure occurs in the $i^{th}$ position, then an easy calculation leads to
    \begin{align*}
    \left\|E_{\Lambda,F,G}\right\|_\mathcal{F} = \sqrt{tr\left(\left(\Theta_{F}^*D\Theta_G\Theta_{G}^*\right)\left(D\Theta_F\right)\right)} = \sqrt{tr\left(D\Theta_F\Theta_{F}^*D\Theta_G\Theta_{G}^*\right)}  = \sqrt{\|f_i\|^2\|g_i\|^2} = \|f_i\|\;\|g_i\|
\end{align*}
and hence,
\begin{align}
\epsilon_{F,G}^{(1)} &=  \max\limits_{1\leq i \leq N} \|f_i\|\;\|g_i\|\label{e1fg formula}\\
&\geq \frac{n}{N}.\nonumber
\end{align}
 From
\cite[Theorem 2.1]{cass2}, by taking $S = I_{n \times n}$  and $a_i = \sqrt{\frac{n}{N}}$, $1 \leq i \leq N$, it can be inferred that there exists an equal norm Parseval frame $F'=\{f'_i\}$ with $\|f'_i\|=\sqrt{\frac{n}{N}}$. Now,
   $\epsilon_{F',S_{F'}^{-1}F'}^{(1)}= \epsilon_{F',F'}^{(1)} = \max\limits_{1\leq i \leq N}\|f'_i\|^2 = \frac{n}{N}.$
Hence, $\epsilon^{(1)} = \frac{n}{N}.$
\par
Now, from the above argument, we may write
 \begin{align*}
  {\mathcal{F}}^{(1)} = \left\{(F,G) :  \max\limits_{1 \leq i \leq N} \|f_i\|\;\|g_i\|= \frac{n}{N} \right\}.
\end{align*}
 Let $(F,G) \in {\mathcal{F}}^{(1)}. $ We know that $\|f_i\|\;\|g_i\|\leq \frac{n}{N}$, $\forall\, i$. If $\|f_j\|\,\|g_j\| < \frac{n}{N},$ for some $j,$ then $n = \sum\limits_{i=1}^N \langle f_i, g_i \rangle \leq \sum\limits_{i=1}^N \|f_i\|\,\|g_i\| < n,$ which is not possible. Therefore, $\|f_i\|\,\|g_i\| = \frac{n}{N},$ for $1 \leq i \leq N.$ Hence, we can conclude that $\mathcal{F}^{(1)} = \left\{(F,G) : \|f_i\|\,\|g_i\| = \frac{n}{N},\; 1 \leq i \leq N  \right\}.$\\
\noindent
(ii) Suppose $F$ is a tight frame, with the optimal frame bound $A$, such that  $(F, S_{F}^{-1}F) \in {\mathcal{F}}^{(1)}. $ Then, by (i), $ \|f_i\|^2 = \frac{An}{N},\,\forall i$. Hence, $F$ is an equal norm frame. Conversely, suppose $F$ is a tight frame and for some constant $c$, $\|f_i\| =c ,\;\forall i$. In view of  (i) and taking $A$ as earlier, we need to only show that $c^2 = \frac{An}{N}$, in order to obtain $(F, S_{F}^{-1}F) \in {\mathcal{F}}^{(1)}$. Taking $G$ to be the canonical dual of $F$ in \eqref{eqn2point1},  we have $n=\sum\limits_{i=1}^N \frac{1}{A}\|f_i\|^2 = \frac{N}{A}c^2.$  Therefore, $c= \sqrt{\frac{An}{N}}$, which proves (ii).
\end{proof}

From the above proof, we observe the following
\begin{rem}\label{equal norm PF exist}\quad

\begin{enumerate}
\item [{\em (i)}] There will always exist an equal norm Parseval frame for $\mathcal{H}$.
\item [{\em (ii)}] If $F=\{f_i\}_{i=1}^N$ is an equal norm tight frame with the optimal frame bound $A$, then $\|f_i\|=\sqrt{\frac{An}{N}}$, $\forall$ $1\leq i\leq N$.
\end{enumerate}
\end{rem}
Hence, we arrive at the result below.
\begin{cor}
The set ${\mathcal{F}}^{(1)}$ is nonempty.
\end{cor}

The next result delineates an important attribute of a dual pair that is $1-$erasure optimal under Frobenius norm.

\begin{cor}\label{thm5point4}
Every element in $ {\mathcal{F}}^{(1)}$ is a $1-$uniform dual pair for $\mathcal{H}.$
\end{cor}

\begin{proof}
Let $(F,G) \in  {\mathcal{F}}^{(1)}.$ Then,
$ n= \sum\limits_{i=1}^N \langle f_i, g_i \rangle \leq \sum\limits_{i=1}^N \left|\langle f_i, g_i \rangle \right| \leq \sum\limits_{i=1}^N \|  f_i\|\,\| g_i\| =n ,$ by Theorem \ref{prop3point2}. Utilizing the fact that $\left|\langle f_i, g_i \rangle \right| \leq \|  f_i\|\,\| g_i\| , $ we can deduce that $\left|\langle f_i, g_i \rangle \right| = \|  f_i\|\,\| g_i\| = \frac{n}{N}, \forall\,i.$ All that is left is to show that the above relation holds even without the modulus. Let $\langle f_j, g_j \rangle = a_j + ib_j,\;1 \leq j \leq N$, where $a_j, b_j\in\mathbb{R}$. Then, $ n= \sum\limits_{j=1}^N \langle f_j, g_j \rangle$ implies $\sum\limits_{j=1}^N a_j =n, \sum\limits_{j=1}^N b_j= 0 $ and $\left|\langle f_j, g_j \rangle \right| = \frac{n}{N}$ implies $\sqrt{a_j^2 + b_j^2} = \frac{n}{N},$ for  $1\leq j \leq N.$ Therefore, $\sum\limits_{j=1}^N a_j = \sum\limits_{j=1}^N \sqrt{a_j^2 + b_j^2}, $ which implies that $b_j =0,\,\forall\,j.$ If any $a_j < \frac{n}{N},$ then there is at least one $\ell$ such that $a_\ell > \frac{n}{N}.$ Then, $|\langle f_\ell, g_\ell \rangle|=\sqrt{a_{\ell}^2 + b_{\ell}^2} > \frac{n}{N},$ which is not possible. Therefore, $a_j = \frac{n}{N}$ and $b_j = 0$ for  $1 \leq j \leq N.$ Consequently, $\langle f_j, g_j \rangle = \frac{n}{N},$ for  $1 \leq j \leq N.$
\end{proof}

Further, for a fixed tight frame, if the frame and its canonical dual is $1-$erasure optimal, then it turns out that this pair is indeed unique, as given below.

\begin{cor}
If $F$ is an equal norm tight frame for $\mathcal{H}$ and $(F,G)$ is a $1-$erasure optimal dual pair,  then $G = S_{F}^{-1}F.$
\end{cor}
\begin{proof}
Let $F$ be an equal norm tight frame with the optimal frame bound $A$ and $G$ be a dual of $F$ such that $(F, G)\in \mathcal{F}^{(1)}$. We know that $G=\{g_i\}_{i=1}^N=\left\{\frac{1}{A}f_i + u_i \right\}_{i=1}^N$ for some $U=\left\{u_i \right\}_{i=1}^N$ satisfying  $\sum\limits_{i=1}^{N}\langle f, f_{i} \rangle u_i=0, \; \forall\; f\in {\cal H}.$  From Theorem \ref{prop3point2},  we obtain
$\|f_i\|\left\| \frac{1}{A}f_i + u_i \right\|= \frac{n}{N}$ for all $1 \leq i \leq N.$  Further, by Remark \ref{equal norm PF exist}, we have $\|f_{i}\|=\sqrt{\frac{An}{N}}$  for every $i,$ and so we get $\left\| \frac{1}{A}f_i + u_i \right\| = \sqrt{\frac{n}{AN}},\;1 \leq i \leq N.$ Equivalently, $\frac{1}{A^2}\|f_i\|^2 + \frac{2}{A} Re\langle f_i, u_i \rangle + \|u_i\|^2 = \frac{n}{AN},\;1 \leq i \leq N,$ and hence $\frac{1}{A^2}\sum\limits_{1 \leq i \leq N} \|f_i\|^2 + \frac{2}{A} \sum\limits_{1 \leq i \leq N} Re\langle f_i, u_i \rangle + \sum\limits_{1 \leq i \leq N} \|u_i\|^2  = \frac{n}{A}.$ Consequently,
 \begin{align}\label{equation9}
 	\frac{2}{A} \sum\limits_{1 \leq i \leq N} Re\langle f_i, u_i \rangle + \sum\limits_{1 \leq i \leq N} \|u_i\|^2 = 0.
 \end{align}
Now,  $\sum\limits_{1\leq i \leq N} \langle f,f_i \rangle u_i =0,$ for all $f \in \mathcal{H}.$ In other words, $\Theta_{U}^* \Theta_F = 0.$  Therefore, $0 = tr(\Theta_{U}^* \Theta_F) = tr(\Theta_F \Theta_{U}^*) = \sum\limits_{1\leq i \leq N} \langle u_i, f_i \rangle $, and so from  \eqref{equation9}, we arrive at  $u_i = 0,$ for all $1 \leq i \leq N.$ Thus, $G = S_{F}^{-1}F $.
\end{proof}

Now, we shall give a characterization for an optimal dual pair when there are a higher number of erasures.

\begin{thm}\label{thm3point5}
 For $1 < m \leq N,\,$ $$\epsilon^{(m)} \geq \sqrt{m\left(\frac{n}{N}\right)^2 + \frac{nm(m-1)(N -n)}{N^2(N-1)}}.$$ Further, if there exists a dual pair $(F',G')\in {\mathcal{F}}^{(1)}$ such that  $Re\left(\langle g'_i, g'_j \rangle \langle f'_j, f'_i \rangle\right) $ is a constant, for all $i \neq j,$ then $\epsilon^{(m)} = \sqrt{m\left(\frac{n}{N}\right)^2 + \frac{nm(m-1)(N -n)}{N^2(N-1)}}$. Further,  $${\mathcal{F}}^{(m)}=\left\{ (F,G) \in {\mathcal{F}}^{(1)}: \text{there exists a constant $c$ such that } Re\left(  \langle g_i, g_j \rangle \langle f_j, f_i \rangle\right) = c, \forall \, i \neq j\right\}.$$
\end{thm}
In order to prove the above theorem, we make use of the following lemma.

\begin{lem}\label{lemma3point4}
	Let $\{a_{ij}:i,j=1,2,\ldots,N,\;i>j\}$ be a finite sequence of real numbers and for a fixed $m\in {\mathbb{N}},$ $1\le m\le N,$ let ${\cal A}_m$ be as in Section \ref{section2}. Then, $$\max\limits_{\Lambda\in {\cal A}_m}\;\; \sum_{\substack{j,k=1 \\ j>k}}^m a_{i_ji_k} \geq \frac{m(m-1)S}{N(N-1)},$$ where $S = \sum\limits_{\substack{i,j=1 \\ i>j}}^N a_{ij}.$
\end{lem}

\begin{proof}
	Suppose $\max\limits_{\Lambda\in {\cal A}_m}\;\; \sum\limits_{\substack{j,k=1 \\ j>k}}^m a_{i_ji_k} < \frac{m(m-1)S}{N(N-1)}.$ Then, $\sum\limits_{\substack{j,k=1 \\ j>k}}^m a_{i_ji_k} < \frac{m(m-1)S}{N(N-1)},\,\forall \, \Lambda\in {\cal A}_m.$ Taking the sum on both sides, we have $\sum\limits_{\Lambda\in {\cal A}_m} \sum\limits_{\substack{j,k=1 \\ j>k}}^m a_{i_ji_k}< \frac{m(m-1)S}{N(N-1)}|{\cal A}_m|$. As $|{\cal A}_m|={{N}\choose{m}}$, it then implies that ${{N-2}\choose{m-2}}\sum\limits_{\substack{i,j=1 \\ i>j}}^N a_{ij}$ $< \frac{m(m-1)S}{N(N-1)}{{N}\choose{m}}={{N-2}\choose{m-2}}S$. This is not possible and hence, the lemma is proved.
\end{proof}

\begin{proof}[\textbf{Proof of Theorem \ref{thm3point5}}]
  Let $1 <m \leq N$ and  $(F,G) \in {\mathcal{F}}^{(m-1)}$. By Theorem \ref{prop3point2}, we have $\|f_i\|\,\|g_i\| = \frac{n}{N},\;\text{for}\, 1 \leq i \leq N. $
  Suppose erasures occur at the locations $i_1, i_2, \cdots ,i_m$, i.e., $\Lambda=\{i_1, i_2, \cdots ,i_m\}$. Then,
  \begin{align*} 
     \left\|E_{\Lambda, F,G}\right\|^2_\mathcal{F} &= \left\|\Theta_{G}^*D\Theta_F\right\|^2_\mathcal{F} \nonumber \\
     &=  tr(D\Theta_F\Theta_{F}^*D\Theta_G\Theta_{G}^*)  \nonumber \\
     &= \sum_{k=1}^m \langle g_{i_1}, g_{i_k} \rangle \langle f_{i_k}, f_{i_1} \rangle + \sum_{k=1}^m \langle g_{i_2}, g_{i_k} \rangle \langle f_{i_k}, f_{i_2} \rangle + \cdots + \sum_{k=1}^m \langle g_{i_m}, g_{i_k} \rangle \langle f_{i_k}, f_{i_m} \rangle \nonumber\\
     &= \sum_{j=1}^m \left\| g_{i_j} \right\|^2\left\| f_{i_j} \right\|^2 + \sum_{\substack{j,k=1 \\ j\neq k}}^m \langle g_{i_j}, g_{i_k} \rangle \langle f_{i_k}, f_{i_j} \rangle\nonumber\\
     &=m\left(\frac{n}{N}\right)^2 + 2 Re \left( \sum_{\substack{j,k=1 \\ j> k }}^m \langle g_{i_j}, g_{i_k} \rangle \langle f_{i_k}, f_{i_j} \rangle  \right).
  \end{align*}
  Therefore,
  \begin{align} \label{eqn3point7}
      \epsilon_{F,G}^{(m)} = \max\limits_{\Lambda\in\mathcal{A}_m} \left\|E_{\Lambda, F,G}\right\|_\mathcal{F} = \sqrt{m\left( \frac{n}{N} \right)^2 + \max\limits_{\Lambda\in {\cal A}_m} 2Re\left( \sum_{\substack{j,k=1 \\ j> k \\ }}^m \langle g_{i_j}, g_{i_k} \rangle \langle f_{i_k}, f_{i_j} \rangle \right)}.
  \end{align}
A lower estimate for $\max\limits_{\Lambda\in {\cal A}_m}2Re\left( \sum\limits_{\substack{j,k=1 \\ j> k \\ }}^m \langle g_{i_j}, g_{i_k} \rangle \langle f_{i_k}, f_{i_j} \rangle \right)$ can be obtained by appealing to Lemma \ref{lemma3point4}.
  It can be easily seen that $\sum\limits_{i,j=1}^N \langle g_i, g_j \rangle \langle f_j, f_i \rangle$ $= \sum\limits_{i=1}^N \left\langle g_i,\sum\limits_{j=1}^N \langle f_i, f_j \rangle g_j \right\rangle$ $= \sum\limits_{i=1}^N \langle g_i ,f_i \rangle =n. $ So,
  \begin{align}\label{equation3point8}
        2 Re \displaystyle{\left( \sum_{\substack{ i, j =1\\ i >j}}^N \langle g_i, g_j \rangle \langle f_j, f_i \rangle  \right)} = \sum\limits_{i \neq j}\langle g_i, g_j \rangle \langle f_j, f_i \rangle = n - \frac{n^2}{N},
  \end{align}
   and therefore $\max\limits_{\Lambda\in {\cal A}_m} 2Re\left( \sum\limits_{\substack{j,k=1 \\ j> k \\ }}^m \langle g_{i_j}, g_{i_k} \rangle \langle f_{i_k}, f_{i_j} \rangle \right) \geq  \frac{nm(m-1)(N -n)}{N^2(N-1)}.$  From \eqref{eqn3point7}, we then obtain
  $$\epsilon_{F,G}^{(m)} \geq  \sqrt{m\left(\frac{n}{N} \right)^2 + \frac{nm(m-1)(N -n)}{N^2(N-1)}}.$$  Thus, $\epsilon^{(m)} \geq \sqrt{m\left(\frac{n}{N} \right)^2 + \frac{nm(m-1)(N -n)}{N^2(N-1)}}.$
\par
Now, we shall assume the existence of a dual pair $(F', G') \in \mathcal{F}^{(1)}$ such that  $Re\left(\langle g'_i, g'_j \rangle \langle f'_j, f'_i \rangle\right) = c,$  for all $i \neq j$ and for some constant $c.$ Then, by  \eqref{equation3point8},   $c(N^2 - N) = n - \frac{n^2}{N},$ which implies $c= \frac{n(N -n)}{N^2(N-1)}.$ Let $\ell\in \mathbb{N}, \, 1<\ell\le m$. Substituting the value of $c$ in \eqref{eqn3point7}, we get $\epsilon_{F',G'}^{(\ell)} = \sqrt{\ell\left( \frac{n}{N} \right)^2 + \ell(\ell-1)c} = \sqrt{\ell\left(\frac{n}{N} \right)^2 + \frac{n\ell(\ell-1)(N -n)}{N^2(N-1)}}.$ As $(F', G') \in \mathcal{F}^{(\ell-1)}$, we obtain $\epsilon^{(\ell)} = \sqrt{\ell\left(\frac{n}{N}\right)^2 + \frac{n\ell(\ell-1)(N -n)}{N^2(N-1)}}$ and $(F', G') \in \mathcal{F}^{(\ell)}.$ Therefore,
  \begin{align*}
     \left\{ (F,G) \in {\mathcal{F}}^{(1)}: \text{there exists a constant $c$ such that } Re\left(\langle g_i, g_j \rangle \langle f_j, f_i \rangle\right)   = c, \forall \, i \neq j \right\}\subset {\mathcal{F}}^{(m)}.
  \end{align*}
Now, suppose $(F,G) \in {\mathcal{F}}^{(m)}$. Using \eqref{eqn3point7} and the value of $\epsilon^{(m)}$, we may deduce that
\begin{align*}
\max\limits_{\Lambda\in {\cal A}_m} 2Re\left( \sum_{\substack{j,k=1 \\ j> k \\ }}^m \langle g_{i_j}, g_{i_k} \rangle \langle f_{i_k}, f_{i_j} \rangle \right)
=\frac{nm(m-1)(N -n)}{N^2(N-1)},
\end{align*}
which further implies that
\[2Re\left( \sum_{\substack{j,k=1 \\ j> k \\ }}^m \langle g_{i_j}, g_{i_k} \rangle \langle f_{i_k}, f_{i_j} \rangle \right)
\leq\frac{nm(m-1)(N -n)}{N^2(N-1)},\,\forall\,\Lambda\in {\cal A}_m. \]
Now, suppose the above inequality is strict for some $\Lambda_{0}\in {\cal A}_m$. Then, taking the sum over all members of ${\cal A}_m$, we get
\begin{align*}
\sum_{\Lambda\in {\cal A}_m}2Re\left( \sum_{\substack{j,k=1 \\ j> k \\ }}^m \langle g_{i_j}, g_{i_k} \rangle \langle f_{i_k}, f_{i_j} \rangle \right)
< \frac{nm(m-1)(N -n)}{N^2(N-1)}|{\cal A}_m|=\frac{n(N-n)}{N}{{N-2}\choose{m-2}}.
\end{align*}
However,
\begin{align*}
\sum_{\Lambda\in {\cal A}_m}2Re\left( \sum_{\substack{j,k=1 \\ j> k \\ }}^m \langle g_{i_j}, g_{i_k} \rangle \langle f_{i_k}, f_{i_j} \rangle \right)
&= {{N-2}\choose{m-2}}2Re\left( \sum_{\substack{i,j=1\\ i> j \\ }}^N \langle g_{i}, g_{j} \rangle \langle f_{j}, f_{i} \rangle \right)
={{N-2}\choose{m-2}}\frac{n(N-n)}{N},
\end{align*}
using \eqref{equation3point8}. This contradiction proves that
 \begin{align}\label{eqn11}
  \sum_{\substack{j,k=1 \\ j\neq k \\ }}^m \langle g_{i_j}, g_{i_k} \rangle \langle f_{i_k}, f_{i_j} \rangle
=\frac{nm(m-1)(N -n)}{N^2(N-1)},\,\forall\,\Lambda\in {\cal A}_m.
\end{align}
We observe that \eqref{eqn11} in fact holds for every $\ell\in \mathbb{N}, \,1<\ell\le m$
and making use of the expression for $\ell=2$, we may conclude that \[Re\left(\langle g_{i}, g_{j} \rangle \langle f_{j}, f_{i} \rangle\right)=\frac{\langle g_{i}, g_{j} \rangle \langle f_{j}, f_{i} \rangle + \langle g_{j}, g_{i} \rangle \langle f_{i}, f_{j} \rangle}{2}	=\frac{n(N -n)}{N^2(N-1)},\;\forall i \neq j.\]
Therefore,
 \begin{align*}
 {\mathcal{F}}^{(m)} = 	\left\{ (F,G) \in {\mathcal{F}}^{(1)}: \text{there exists a constant $c$ such that }  Re\left( \langle g_i, g_j \rangle\langle f_j, f_i \rangle \right) = c, \forall \, i \neq j\right\}.\qquad\qedhere
\end{align*}

\end{proof}
An easy consequence of Theorems \ref{prop3point2} and \ref{thm3point5} is the following
result, which gives a necessary and  sufficient condition for  a pair of a tight frame and its canonical dual to be an m$-$erasure optimal dual pair for every $m$.

\begin{thm}
Let $F$ be a tight frame for $\mathcal{H}.$ Suppose there exists a dual pair $(F',G') \in \mathcal{F}^{(1)}$ such that $Re\left( \langle g'_i, g'_j \rangle\langle f'_j, f'_i \rangle\right)=c $ for all $i \neq j,$ and for   some constant $c$. Then, $(F,S_{F}^{-1}F)$ is an $m$-erasure optimal dual pair for every $m \in \{1,2,\ldots,N\}$ if and only if $(F, S_{F}^{-1}F)$ is a $2-$uniform dual pair or equivalently, $F$ is an equiangular frame.
\end{thm}

\begin{thm}
	Let $(F,G) $ be a dual pair in $\mathcal{H}$ and $U$ be a unitary operator on $\mathcal{H}.$ Then, $(F,G) \in \mathcal{F}^{(m)}$ if and only if $(UF,UG) \in  \mathcal{F}^{(m)}.$
\end{thm}

\begin{proof}
Using \eqref{e1fg formula} and  \eqref{eqn3point7}, it is easy to observe that
	\begin{align*}
		\epsilon_{F,G}^{(1)} = \max\limits_{1 \leq i \leq N} \|f_i\|\,\|g_i\| = \max\limits_{1 \leq i \leq N} \|Uf_i\|\,\|Ug_i\| = \epsilon_{UF,UG}^{(1)}
	\end{align*}
and for $1 <m \leq N,$
\begin{align*}
	\epsilon_{F,G}^{(m)} &= \sqrt{m\left( \frac{n}{N} \right)^2 + \max\limits_{\Lambda\in {\cal A}_m} 2Re\left( \sum_{\substack{j,k=1 \\ j> k \\ }}^m \langle g_{i_j}, g_{i_k} \rangle \langle f_{i_k}, f_{i_j} \rangle \right)} \\ &= \sqrt{m\left( \frac{n}{N} \right)^2 + \max\limits_{\Lambda\in {\cal A}_m} 2Re\left( \sum_{\substack{j,k=1 \\ j> k \\ }}^m \langle Ug_{i_j}, Ug_{i_k} \rangle \langle Uf_{i_k}, Uf_{i_j} \rangle \right)} \\&= \epsilon_{UF,UG}^{(m)}.
\end{align*}
The result then follows.
\end{proof}

\section{Optimality analysis with respect to spectral radius}
Here, the spectral radius of the error operator is taken to be the measure of optimality for the analysis of  optimal  dual pairs.  For an $(N,n)$ dual pair $(F,G),$  we define
$r_{F,G}^{(m)}:=$ $\max \left\{ \rho( E_{\Lambda,F,G} ) : \Lambda\in\mathcal{A}_m \right\}$, $1 \leq m \leq N,$ where $\rho$ denotes the spectral radius. Now, let
\begin{align*}
& r^{(1)} := \inf \bigg\{ r_{F,G}^{(1)} : (F,G)\; \text{is an}\; (N,n)\; \text{dual pair for }\; \mathcal{H} \bigg\}, \\
&\mathcal{R}^{(1)} := \left\{ (F,G) : r_{F,G}^{(1)} = r^{(1)} \right\}, \mbox{ and}\\
\hspace*{-4in}\mbox{ for } 1<m\le N,&\\
&r^{(m)} := \inf \left\{ r_{F,G}^{(m)} : (F,G) \in \mathcal{R}^{(m-1)} \right\},\\& \mathcal{R}^{(m)} := \left\{ (F,G)  \in \mathcal{R}^{(m-1)} : r_{F,G}^{(m)} = r^{(m)} \right\}.
\end{align*}
The dual pairs in ${\mathcal{R}}^{(m)}$ are said to be $m-$erasure spectrally optimal dual pairs. It is easy to see that $r_{F,G}^{(1)} = \max\limits_{1 \leq i \leq N} \left| \langle f_i , g_i \rangle \right| $. Also, using  \eqref{eqn2point1}, we get $\sum\limits_{i=1}^{N} |\langle f_i, g_i \rangle| \geq \left|\sum\limits_{i=1}^{N} \langle f_i, g_i \rangle \right|= n$. It then follows that $r_{F,G}^{(1)} \geq \frac{n}{N}$ and hence, $ r^{(1)} \geq \frac{n}{N}.$ We shall show that $ r^{(1)} = \frac{n}{N}$, by constructing a family of dual pairs $(F,G)$ for which $r_{F,G}^{(1)} = \frac{n}{N}. $ From \cite[Theorem 3.9]{adm}, the existence of a $1-$erasure spectrally optimal dual pair of the form $(F,F)$, where $F$ is a Parseval frame, can be shown. Here, an explicit construction of a $1-$erasure spectrally optimal dual pair $(F,G)$, where $F$ is not a Parseval frame, has been provided.

\begin{thm}\label{lem4point2}
 There exists an $(N,n)$ dual pair $(F,G)$  for ${\cal H}$ such that  $r_{F,G}^{(1)} = \frac{n}{N}. $
\end{thm}

\begin{proof}
Let   $\{ e_1,e_2,\ldots,e_n \} $ be an orthonormal basis  of $\mathcal{H}.$ If $ N > n,$ then we extend this orthonormal basis to a frame $F = \{f_i\}_{i=1}^N$  for $\mathcal{H}$ by defining
\[
f_i := \begin{cases} e_i,\; & 1\leq i \leq n,\\
	\sum\limits_{j=1}^n e_j, \;\;& n+1 \leq i \leq N.
	\end{cases}
\]
Consider the sequence $G=\{g_i\}_{i=1}^N$, where
\begin{eqnarray*}
g_i := \begin{cases} e_i-\frac{N-n}{N}\sum\limits_{j=1}^n e_j,\; & 1\leq i \leq n,\\
	\frac{1}{N}\sum\limits_{j=1}^n e_j, \;\;& n+1 \leq i \leq N.
	\end{cases}
\end{eqnarray*}
For $f\in\mathcal{H}$,
\begin{align*}
\sum_{i=1}^N \langle f,f_i\rangle g_i
&=\sum_{i=1}^n \langle f,e_i\rangle \left(e_i-\frac{N-n}{N}\sum\limits_{j=1}^n e_j\right) + \sum_{i=n+1}^N\left\langle f, \sum_{j=1}^n e_j \right\rangle\left(\frac{1}{N}\sum_{j=1}^n e_j\right)=f,
\end{align*}
which proves that $(F,G)$  is an $(N,n)$ dual pair. Further, it can be easily seen that $\langle f_i,g_i\rangle=\frac{n}{N}$, for $1\leq i\leq N$. So, $r_{F,G}^{(1)} =  \frac{n}{N}.$
  If $ N=n $,  then we take $G$ to be  the canonical dual of $F=\{e_i\}_{i=1}^N,$ which is $F$ itself. Now, $ r_{({F,S_{F}^{-1}F})}^{(1)} = \max\limits_{1 \leq i \leq N} \|f_i  \|^2 =1=\frac{n}{N}.$
\end{proof}
Using the existence of such a dual pair, one can prove the following theorem, which also provides a  characterization for $1-$erasure spectrally optimal dual pairs, along similar lines of proof of \cite[Theorem 3.8]{adm}. Alternatively, we can obtain the result by specifying $q_i=\frac{N}{n},\forall i$ in \cite[Theorem 3.8]{adm}.
\begin{thm}\label{thm4point3}
The value of $r^{(1)}$ is $\frac{n}{N}.$ Moreover, a dual pair $(F,G) $ in $\mathcal{H}$ is $1-$erasure spectrally optimal if and only if it is $1-$uniform.
\end{thm}

Let $(F,G)$ be a dual pair. Then, by \cite[Proposition 1]{dev}, we have $$r_{F,G}^{(2)} = \max\limits_{i \neq j}\frac{1}{2} \left| \alpha_{ii} +\alpha_{jj} \pm \sqrt{(\alpha_{ii} - \alpha_{jj} )^2 + 4\alpha_{ij}\alpha_{ji} }     \right|,$$ where $\alpha_{ij} = \langle f_i , g_j \rangle,  1 \leq i,j \leq N $ and for  a complex number $\alpha+i\beta$ with $\alpha,\beta\in {\mathbb{R}},$ $\sqrt{\alpha+i\beta}=a+ib$, $a\ge 0$. If $(F,G)\in \mathcal{R}^{(1)},$ then
 \begin{align}\label{eqn4point10}
 	r_{F,G}^{(2)} = \max\limits_{i \neq j} \left|\frac{n}{N} + \sqrt{\alpha_{ij}\alpha_{ji} }     \right|,
 \end{align}
  as $\alpha_{ii} = \frac{n}{N},\forall i.$ This expression plays a  very crucial role in our analysis to characterize $2-$erasure spectrally optimal dual pairs, which is discussed below.
   This theorem looks similar to \cite[Theorem 3.11]{adm} with  $q_i=\frac{N}{n},\forall i$. However, that proof works only for real Hilbert spaces. The result discussed below takes care of even complex Hilbert spaces.
\begin {thm}\label{Thm4point4} Suppose $N>n$. Then, the following hold:
\begin{enumerate}
\item [{\em (i)}] $ r^{(2)} \geq \frac{n}{N} + \sqrt{\frac{n(N - n)}{N^2(N-1)}}.$
\item [{\em (ii)}]Suppose that there exists  a $2-$uniform dual pair for ${\cal H}.$ Then,
\begin{enumerate}
\item $ r^{(2)} = \frac{n}{N} + \sqrt{\frac{n(N - n)}{N^2(N-1)}},$ and
\item a dual pair  $(F,G)$ is 2-erasure optimal if and only if  $(F,G)$ is $2-$uniform.
\end{enumerate}
\end{enumerate}
\end{thm}
\begin{proof}
Let $(F,G)\in \mathcal{R}^{(1)}.$  Then,
\begin{align}
r_{F,G}^{(2)} =  \max\limits_{i \neq j} \bigg| \frac{n}{N} +\sqrt{\langle f_i , g_j \rangle \langle f_j , g_i \rangle} \bigg|
 \ge \bigg| \frac{n}{N} +\sqrt{\langle f_{i_0} , g_{j_0} \rangle \langle f_{j_0} , g_{i_0} \rangle} \bigg|,\label{E:new4point2}
\end{align}
where $i_0\neq j_0$ is a pair of indices for which $\max\limits_{i \neq j}Re\left(\langle f_i , g_j \rangle \langle f_j , g_i \rangle\right)= Re\left(\langle f_{i_0} , g_{j_0} \rangle \langle f_{j_0} , g_{i_0} \rangle\right)$. We observe that  $Re\left(\langle f_{i_0} , g_{j_0} \rangle \langle f_{j_0} , g_{i_0} \rangle\right)$ $>0.$
In fact, using Theorem \ref{thm4point3}, we have
\begin{align}\label{eqn4point12}
 \sum\limits_{\substack{i, j =1\\i\neq j}}^N  \langle f_i , g_j \rangle \langle f_j , g_i \rangle&=
\sum\limits_{i, j =1}^N  \langle f_i , g_j \rangle \langle f_j , g_i \rangle -\sum\limits_{i =1}^N
\langle f_i, g_i\rangle^2
=\frac{n(N-n)}{N}>0.
\end{align}
 So, $\sum\limits_{\substack{i, j =1\\i\neq j}}^N  Re\left(\langle f_i , g_j \rangle \langle f_j , g_i \rangle\right) >0$, which implies that $\max\limits_{i \neq j}Re\left(\langle f_i , g_j \rangle \langle f_j , g_i \rangle\right)$ cannot be negative or zero. Taking $\langle f_{i_0} , g_{j_0} \rangle \langle f_{j_0} , g_{i_0} \rangle =c_0e^{i\theta_0},$ where $c_0\ge 0$ and  $-\pi< \theta_0\le\pi,$ we obtain $c_0\cos \theta_0>0$. Therefore, we actually have $c_0>0$ and $-\frac{\pi}{2}< \theta_0 < \frac{\pi}{2}$, and hence $\cos\left(\frac{\theta_0}{2}\right)\ge\sqrt{\cos(\theta_0)}$. Now, from \eqref{E:new4point2}, we get
\begin{align}\label{E:eqn4point4}
r_{F,G}^{(2)}
  &\ge  \frac{n}{N} + Re\sqrt{\langle f_{i_0} , g_{j_0} \rangle \langle f_{j_0} , g_{i_0} \rangle}=\frac{n}{N}+\sqrt{c_0}\cos\left(\frac{\theta_0}{2}\right)
  \ge \frac{n}{N} + \sqrt{Re\left(\langle f_{i_0} , g_{j_0} \rangle \langle f_{j_0} , g_{i_0}\rangle\right)}.
\end{align}
Next, as  $Re\left(\langle f_i , g_j \rangle \langle f_j , g_i \rangle\right) \le  Re\left(\langle f_{i_0} , g_{j_0} \rangle \langle f_{j_0} , g_{i_0}\rangle\right)$ for all $ i\neq j$, we obtain
 $\sum\limits_{\substack{i, j =1\\i\neq j}}^N Re\left(\langle f_i , g_j \rangle \langle f_j , g_i \rangle\right)$ $\le  N(N-1)Re\left(\langle f_{i_0} , g_{j_0} \rangle \langle f_{j_0} , g_{i_0}\rangle\right)$ and so,
 $\frac{n(N-n)}{N^2(N-1)}\le Re\left(\langle f_{i_0} , g_{j_0} \rangle \langle f_{j_0} , g_{i_0}\rangle\right)$.
This proves that   $ r^{(2)}\geq \frac{n}{N} + \sqrt{\frac{n(N - n)}{N^2(N-1)}}.$
\par
 Suppose  $(F', G')$ is a $2-$uniform dual pair for ${\cal H}.$ Then,  $(F', G')$ is $1-$uniform and hence by  Theorem \ref{thm4point3},  it is 1-erasure optimal. Now,
using  \eqref{eqn4point12} and the $2-$uniform condition,  we obtain $\langle f'_i,g'_j \rangle \langle f'_j,g'_i \rangle = \frac{n(N - n)}{N^2(N-1)}$ for all $i \neq j$. Therefore, in view of \eqref{eqn4point10} and the first part of the theorem, we get $r^{(2)}=r_{F',G'}^{(2)} = \frac{n}{N} + \sqrt{\frac{n(N - n)}{N^2(N-1)}}.$ Also, every other $2-$uniform dual pair belongs to ${\mathcal R}^{(2)}.$

Conversely, suppose   $(F,G)\in {\mathcal R}^{(2)}.$   Then, $(F,G)\in {\mathcal R}^{(1)}$ and hence, by Theorem \ref{thm4point3}, it is $1-$uniform.
In order to show the $2-$uniform condition, we first recall from the earlier part of the proof that $\max\limits_{i \neq j}Re\left(\langle f_i , g_j \rangle \langle f_j , g_i \rangle\right)\ge \frac{n(N-n)}{N^2(N-1)}$ and observe from  \eqref{eqn4point12} that $0=\sum\limits_{\substack{i, j =1\\i\neq j}}^N Im\left( \langle f_i , g_j \rangle \langle f_j , g_i \rangle\right)$ $\le N(N-1)\max\limits_{i \neq j}Im\left(\langle f_i , g_j \rangle \langle f_j , g_i \rangle\right)$, which in turn implies that
$\max\limits_{i \neq j}Im\left(\langle f_i , g_j \rangle \langle f_j , g_i \rangle\right)\ge 0.$
Now, suppose $\max\limits_{i \neq j} Re\left(\langle f_i , g_j \rangle \langle f_j , g_i \rangle\right)> \frac{n(N - n)}{N^2(N-1)}.$ Then, by \eqref{E:eqn4point4}, we have
\begin{align*}	
r_{F,G}^{(2)} \geq \frac{n}{N} + \sqrt{\max\limits_{i \neq j} Re\left(\langle f_i , g_j \rangle \langle f_j , g_i \rangle\right)}
 > \frac{n}{N} + \sqrt{\frac{n(N - n)}{N^2(N-1)}}=r^{(2)},
\end{align*}
which is not possible. So then, we infer that  $\max\limits_{i \neq j} Re\left(\langle f_i , g_j \rangle \langle f_j , g_i \rangle\right) = \frac{n(N - n)}{N^2(N-1)}.$ Moreover, we would also obtain
$Re\left(\langle f_i , g_j \rangle \langle f_j , g_i \rangle\right) = \frac{n(N - n)}{N^2(N-1)}$ for all $i\neq j$, for if
$Re\left(\langle f_i , g_j \rangle \langle f_j , g_i \rangle\right) < \frac{n(N - n)}{N^2(N-1)}$ for some $i\neq j$, then it would contradict \eqref{eqn4point12}. Next, suppose we assume $0<\max\limits_{i \neq j} Im\left(\langle f_i , g_j \rangle \langle f_j , g_i \rangle\right)=Im\left(\langle f_{i_1} , g_{j_1} \rangle \langle f_{j_1} , g_{i_1} \rangle\right)$, for some $i_1\neq j_1$ and take $\langle f_{i_1} , g_{j_1} \rangle \langle f_{j_1} , g_{i_1} \rangle=c_1e^{i\theta_1},$ with $c_1\ge 0$ and $-\pi <  \theta_1 \le \pi$. Then,  we deduce that $c_1\cos{\theta_1}=\frac{n(N - n)}{N^2(N-1)}>0$ and
$c_1\sin{\theta_1}>0,$ which leads to $c_1>0$ and $0<  \theta_1 < \frac{\pi}{2}.$  By a similar argument as above and using the fact that $\cos(\frac{\theta_1}{2})>\sqrt{\cos(\theta_1)}$, we get the contradiction  $r_{F,G}^{(2)}>r^{(2)}.$ Therefore, $\max\limits_{i \neq j} Im\left(\langle f_i , g_j \rangle \langle f_j , g_i \rangle\right)=0$, which in turn gives $Im\left(\langle f_i , g_j \rangle \langle f_j , g_i \rangle\right) =0$ for all $i\neq j$, following the lines of proof of  $Re\left(\langle f_i , g_j \rangle \langle f_j , g_i \rangle\right)=\frac{n(N - n)}{N^2(N-1)}$ for all $i\neq j$. We may then conclude that  $\langle f_i , g_j \rangle \langle f_j , g_i \rangle=  \frac{n(N - n)}{N^2(N-1)},$ $\forall i\neq j$, thereby proving the theorem.
\end{proof}
From the proof of the above theorem, we also have the following
\begin{rem}
 For $(F,G)\in \mathcal{R}^{(1)},$ $r_{F,G}^{(2)} = \frac{n}{N} + \sqrt{\frac{n(N - n)}{N^2(N-1)}}$ if and only if $(F,G)$ is $2-$uniform.
\end{rem}

\begin{rem}
When $N=n$, every dual pair $(F,G)$ belongs to $\mathcal{R}^{(2)}.$ This can be proved as is done in \cite{adm},
by taking $q_i=1$ for all $i$.
\end{rem}

The following is a characterization, in terms of  equiangular frames, of an optimal dual pair consisting of a frame and its canonical dual, when there are two erasures.
\begin{cor}
Let $F = \{f_i\}_{i=1}^{N} $ be a frame for $\mathcal{H}$. Then the following hold:
\begin{enumerate}
\item [{\em (i)}]
If   $\left\{S_{F}^{-\frac{1}{2}} f_i \right\}_{i=1}^N $ is an equiangular frame, then $(F,S_{F}^{-1}F)$ is a $2-$erasure spectrally optimal dual pair.
\item [{\em (ii)}] If $(F,S_{F}^{-1}F)$ is a $2-$erasure spectrally optimal dual pair and there exists a $2-$uniform dual pair for ${\cal H }$, then $\left\{S_{F}^{-\frac{1}{2}} f_i \right\}_{i=1}^N $ is an equiangular frame for ${\cal H}$.
    \end{enumerate}
\end{cor}
\noindent
The above result follows easily as an application of Theorem \ref{Thm4point4}, by using the relation $\langle f_i,S_{F}^{-1} f_i \rangle = \left\|  S_{F}^{-\frac{1}{2}} f_i  \right\|^2. $
Next, we show that for every $n\in\mathbb{N}$, we can easily construct an $(n+1,n)$ dual pair which is $2-$uniform, and hence $2-$erasure spectrally optimal.

\begin{thm}
There always exists a $2-$erasure spectrally optimal $(n+1,n)$ dual pair for any Hilbert space of dimension $n$.
\end{thm}

\begin{proof}
Consider the dual pair $(F,G)$ constructed as in Theorem \ref{lem4point2}, by taking $N=n+1$. It has been observed that it is $1-$uniform. Further, it also satisfies the $2-$uniform condition for a dual pair, which is an easy computation. The result now follows from Theorem \ref{Thm4point4}.
\end{proof}

\begin{thm}
	Let $(F,G) $ be an (N,n) dual pair for $\mathcal{H}$ and $T$ be an invertible operator on $\mathcal{H}.$ Then, $(F,G) \in \mathcal{R}^{(m)}$ if and only if $(TF, (T^*)^{-1}G) \in  \mathcal{R}^{(m)},$ for $m=1,2 .$
\end{thm}

\begin{proof}
	It is easy to verify that 	$G=\{g_i\}_{i=1}^N$ is a dual frame for the frame $F=\{f_i\}_{i=1}^N$ if and only if $(T^*)^{-1}G =  \{(T^*)^{-1}g_i\}_{i=1}^N$ is a dual frame for the frame $TF = \{Tf_i\}_{i=1}^N.$ Now,
		$r_{F,G}^{(1)} = \max\limits_{1 \leq i \leq N} \left| \langle f_i, g_i \rangle \right| = \max\limits_{1 \leq i \leq N} \left| \langle Tf_i, (T^*)^{-1}g_i \rangle \right| = r_{TF, (T^*)^{-1}G}^{(1)}$
	and hence, $(F,G) \in \mathcal{R}^{(1)}$ if and only if $\left(TF, (T^*)^{-1}G\right) \in  \mathcal{R}^{(1)}.$
	Also, if $(F,G) \in \mathcal{R}^{(1)}$,
				$$r_{F,G}^{(2)} = \max\limits_{i \neq j} \left|\frac{n}{N} + \sqrt{\langle f_i, g_j \rangle \langle f_j, g_i \rangle  }     \right| = \max\limits_{i \neq j} \left|\frac{n}{N} + \sqrt{\langle Tf_i, (T^*)^{-1}g_j \rangle \langle Tf_j, (T^*)^{-1}g_i \rangle }     \right| = r_{TF, (T^*)^{-1}G}^{(2)}.$$
Thus, $(F,G) \in \mathcal{R}^{(2)}$ if and only if $\left(TF, (T^*)^{-1}G\right) \in  \mathcal{R}^{(2)}.$
	\end{proof}

\section{Optimality analysis using numerical radius}

In this section, we analyze the optimality of dual pairs by taking the numerical radius of the error operator as the measure. Let $(F,G)$ be an $(N, n)$ dual pair for  $\mathcal{H}$. For  $1 \leq m \leq N,$ we define $\eta_{F,G}^{(m)} = \max \left\{ \omega\left( E_{\Lambda,F,G} \right)  : \Lambda\in\mathcal{A}_m \right\},$
where $ \omega\left( E_{\Lambda,F,G} \right) := \sup \left\{\left| \langle  E_{\Lambda,F,G} f, f \rangle  \right| : \|f\| = 1\right\}$ is the numerical radius of the error operator $E_{\Lambda,F,G}.$ Also,  we define
\begin{align*}
	&\eta^{(1)} := \inf\bigg\{ \eta_{F,G}^{(1)} : (F,G)\, \text{is a dual pair in $\mathcal{H}$} \bigg\}, \\
& \mathcal{N}^{(1)} := \left\{ (F,G) : \eta_{F,G}^{(1)} = \eta^{(1)} \right\}, \\
& \eta^{(m)} := \inf \left\{ \eta_{F,G}^{(m)} : (F,G) \in \mathcal{N}^{(m-1)} \right\} \text{ for }m>1 \text{ and } \\
& \mathcal{N}^{(m)} := \left\{ (F,G) \in \mathcal{N}^{(m-1)}: \eta_{F,G}^{(m)} = \eta^{(m)} \right\} \text{ for } m>1.
\end{align*}
The members of ${\mathcal{N}}^{(m)}$ are called $m-$erasure numerically  optimal dual pairs. Here, we analyze optimal dual pairs when there is a single erasure, in which case  $\omega\left( E_{\Lambda,F,G} \right)$ can be expressed \cite[Lemma 2.1]{chien} in terms of the trace and operator norm of the error operator. More precisely, $\omega\left( E_{\Lambda,F,G} \right) = \frac{|tr\left( E_{\Lambda,F,G} \right)| + \left\| E_{\Lambda,F,G} \right\|}{2}$. The following theorem provides the value of $\eta^{(1)}$ and some necessary and sufficient conditions for a dual pair to be numerically optimal.
\begin{thm}\label{thm5point1}
    The value of $ \eta^{(1)}$ is $\frac{n}{N}.$ Further, for an $(N, n)$ dual pair $(F,G)$,
\begin{enumerate}
\item [{\em (i)}] $(F,G)\in \mathcal{N}^{(1)}$ if and only if $\|f_i\|\,\|g_i\|= \frac{n}{N},\,\forall \, 1\leq i \leq N$.
\item [{\em (ii)}]If $(F,G)\in \mathcal{N}^{(1)}$, then $\langle f_i,g_i \rangle = \frac{n}{N},\,\forall \,  1\leq i \leq N .$ However, the converse need not hold.
\end{enumerate}
 \end{thm}

\begin{proof}
Let  $(F,G)$ be a dual pair. If the erasure occurs at the $i^{th}$ position, then we can easily compute that $\omega\left( E_{\Lambda,F,G} \right)$ $=\dfrac{|\langle g_i,f_i\rangle| + \|g_i\|\,\|f_i\|}{2}$. Further, $\eta_{F,G}^{(1)} =$ $\max\limits_{1 \leq i \leq N} \dfrac{|\langle g_i,f_i\rangle| + \|g_i\|\,\|f_i\|}{2}$ $\geq \max\limits_{1 \leq i \leq N} |\langle g_i,f_i\rangle|$. As $\left(\max\limits_{1 \leq i \leq N} |\langle g_i,f_i\rangle|\right)N\geq \sum\limits_{i=1}^N |\langle g_i,f_i\rangle| \geq \left|\sum\limits_{i=1}^N \langle g_i,f_i\rangle \right| =n,$ we have $\eta_{F,G}^{(1)}\geq \frac{n}{N}.$ By Remark \ref{equal norm PF exist}, there exists a dual pair $(F, F)$ such that $\eta_{F,F}^{(1)} = \max\limits_{1 \leq i \leq N} \|f_i\|^2 = \frac{n}{N}$, and hence we may conclude that $\eta^{(1)} =\frac{n}{N}.$
\par
Now, consider $(F,G)\in {\cal N}^{(1)}.$ Then, $\max\limits_{1 \leq i \leq N} \dfrac{|\langle g_i,f_i\rangle| + \|g_i\|\,\|f_i\|}{2} = \dfrac{n}{N}.$ If  $\dfrac{|\langle g_j,f_j\rangle| + \|g_j\|\,\|f_j\|}{2} < \dfrac{n}{N}$ for some $j,$  then $\sum\limits_{i=1}^N \dfrac{|\langle g_i,f_i\rangle| + \|g_i\|\,\|f_i\|}{2} <n$. On the other hand, $\sum\limits_{i=1}^N \dfrac{|\langle g_i,f_i\rangle| + \|g_i\|\,\|f_i\|}{2}\geq \sum\limits_{i=1}^N \left|\langle g_i,f_i \rangle \right| \geq \sum\limits_{i=1}^N \langle g_i,f_i \rangle =n$, which is not possible. So,
\begin{equation}\label{E:5point1}
\dfrac{|\langle g_i,f_i\rangle| + \|g_i\|\,\|f_i\|}{2} = \frac{n}{N},\; \text{for all}\; 1 \leq i \leq N.
\end{equation}
Next, $n\leq \sum\limits_{i=1}^N \left|\langle g_i,f_i \rangle \right| \leq \sum\limits_{i=1}^N \dfrac{|\langle g_i,f_i\rangle| + \|g_i\|\,\|f_i\|}{2} = n$ suggests that $\sum\limits_{i=1}^N \left|\langle g_i,f_i \rangle \right| = \sum\limits_{i=1}^N \dfrac{|\langle g_i,f_i\rangle| + \|g_i\|\,\|f_i\|}{2}.$ As $\left|\langle g_i,f_i \rangle \right| \leq \dfrac{|\langle g_i,f_i\rangle| + \|g_i\|\,\|f_i\|}{2},$ we actually have the equality, and so $\left|\langle g_i,f_i \rangle \right| =\|f_{i}\|\,\|g_{i}\|$, $\forall\,i$. Furthermore, using \eqref{E:5point1}, we obtain $\|f_{i}\|\,\|g_{i}\|=  \frac{n}{N} ,\;\forall 1\leq i \leq N.$
\par
Conversely, for an $(N, n)$ dual pair $(F,G)$, $\|f_{i}\|\,\|g_{i}\|= \frac{n}{N} ,\;\forall 1\leq i \leq N$ leads to $\dfrac{|\langle g_i,f_i\rangle| + \|g_i\|\,\|f_i\|}{2}$ $\leq \|f_{i}\|\,\|g_{i}\|=  \dfrac{n}{N}$, $\forall\, i$. This in turn gives $\eta_{F,G}^{(1)} \leq \frac{n}{N}$, and so $(F,G)\in {\cal N}^{(1)},$ thereby proving statement (i) of the theorem.
\par
In order to prove statement (ii), consider $(F,G)\in {\cal N}^{(1)}.$ By statement (i) of this theorem, $\|f_{i}\|\,\|g_{i}\|= \frac{n}{N} ,\;\forall 1\leq i \leq N$. The proof of the consequence $\langle f_i,g_i \rangle = \frac{n}{N},\,\forall \,  1\leq i \leq N$ follows from the proof of Corollary \ref{thm5point4}. We infer that the converse of this statement need not hold, from the example discussed below.
\end{proof}
The following is a counterexample for the converse part mentioned in statement (ii) of the above theorem.
\begin{example}\label{Eg:5.2}
Consider the frame $F= \{f_1,f_2,f_3\}$ for $\mathbb{C}^2,$ where $f_1 = \left[\begin{array}{l}
		1 \\0
	\end{array}\right] $,
	$f_2 = \left[\begin{array}{l}
		0 \\ 1
	\end{array}\right] $
	and $ f_3 = \left[\begin{array}{l}
		1 \\ 1
	\end{array}\right] $, and its canonical dual frame
$\displaystyle{S_{F}^{-1} F = \left\{ \left[\begin{array}{r} \frac{2}{3} \\~\\ -\frac{1}{3} \end{array}\right],  \left[\begin{array}{r} -\frac{1}{3} \\~\\ \frac{2}{3} \end{array}\right], \left[\begin{array}{l} \frac{1}{3} \\~\\ \frac{1}{3}  \end{array}\right] \;\;\right\}}. $ Then, $\langle f_i, S_{F}^{-1}f_i\rangle=\frac{2}{3}=\frac{n}{N}$, $\forall\, i$. However, $\eta_{F,S_{F}^{-1} F}^{(1)} =$ $\max\limits_{1 \leq i \leq N} \dfrac{|\langle S_{F}^{-1}f_i,f_i\rangle| + \|S_{F}^{-1} f_i\|\,\|f_i\|}{2}>\frac{2}{3}=\eta^{(1)},$ which implies that $(F,S_{F}^{-1} F)\notin {\cal N}^{(1)}.$
\end{example}

Given below is an interesting uniqueness type result - if a frame and its canonical dual is $1-$erasure numerically optimal, then no other dual of $F$ can be paired with it for optimality.

\begin{thm} For a frame $F$ in $\mathcal{H}$, if $(F, S_{F}^{-1}F) \in {\cal N}^{(1)},$ then
		$S_{F}^{-\frac{1}{2}}F$ is an equal norm Parseval frame. In addition, if $(F,G) \in \mathcal{N}^{(1)}$ for any dual $G$ of $F$, then $G = S_{F}^{-1}F. $
\end{thm}
 \begin{proof}
Let $(F, S_{F}^{-1}F) \in \mathcal{N}^{(1)},$ where $F$ is any frame. It is well-known that $S_{F}^{-\frac{1}{2}}F$ is a Parseval frame. Further,
\begin{align}\label{E:eqn5.2}
  \left\| S_{F}^{-\frac{1}{2}}f_i \right\|^2 =  \langle f_i, S_{F}^{-1}f_i \rangle =\frac{n}{N},\, 1 \leq i \leq N,
\end{align}
 by Theorem \ref{thm5point1}, which shows that $S_{F}^{-\frac{1}{2}}F$ is an equal norm Parseval frame.
\par
 Suppose, we also assume that $(F,G) \in \mathcal{N}^{(1)}$, where $G$ is a dual of $F$. Then, by Theorem \ref{thm5point1}, we have $\|f_i\|\,\|g_i\|  = \frac{n}{N} = 	\|f_i\|\,\|S_{F}^{-1}f_i\|$ for all $i$, which also suggests that $\|f_i\|\neq 0$ for any $i$. This leads to the relation  $\|g_i\|=\|S_{F}^{-1}f_i\|$ for all $i$. Now, the dual $G$ may be expressed as $G= \left\{ S_{F}^{-1}f_i + u_i  \right\}_{i=1}^N,$ such that $\sum\limits_{i=1}^N\langle f,u_i \rangle f_i =0$, for all $f\in \mathcal{H}$ and in particular, $\sum\limits_{i=1}^N\left\langle S_{F}^{-\frac{1}{2}}f,u_i\right\rangle f_i=0$. We then have $\|S_{F}^{-1}f_i + u_i\|=\|S_{F}^{-1}f_i\|$ and in turn,
 $\|u_i\|^2 + 2 Re \langle u_i,S_{F}^{-1} f_i \rangle =0$, for all $i.$ Further, taking the sum, we get	
\begin{equation} \label{E:eqn5.3}
 \sum\limits_{i=1}^N \|u_i\|^2 + 2 Re \left(\sum\limits_{i=1}^N \left\langle S_{F}^{-\frac{1}{2}}u_i, S_{F}^{-\frac{1}{2}}f_i \right\rangle \right) =0.
\end{equation}
Now, we note that $\sum\limits_{i=1}^N \left\langle f, S_{F}^{-\frac{1}{2}}u_i \right\rangle S_{F}^{-\frac{1}{2}}f_i = S_{F}^{-\frac{1}{2}}\left(\sum\limits_{i=1}^N\left\langle S_{F}^{-\frac{1}{2}}f,u_i\right\rangle f_i  \right) = 0$ for all $f \in \mathcal{H}.$ This implies that $\left\{ S_{F}^{-\frac{1}{2}}f_i + S_{F}^{-\frac{1}{2}}u_i\right\}_{i=1}^N$ is a dual of $\left\{ S_{F}^{-\frac{1}{2}}f_i \right\}_{i=1}^N$. Moreover, using  \eqref{eqn2point1} and \eqref{E:eqn5.2}, we obtain
 $\sum\limits_{i=1}^N \left\langle S_{F}^{-\frac{1}{2}}u_i, S_{F}^{-\frac{1}{2}}f_i \right\rangle=0$. Therefore, from \eqref{E:eqn5.3}, we may conclude that $u_i=0$, for all $i$ and hence, $G = S_{F}^{-1}F. $
\end{proof}

 \begin{thm}
	Let $(F,G) $ be a dual pair in $\mathcal{H}$ and $U$ be a unitary operator on $\mathcal{H}.$ Then, $(F,G) \in \mathcal{N}^{(1)}$ if and only if $(UF,UG) \in  \mathcal{N}^{(1)}.$
\end{thm}

\begin{proof}
Clearly,
	\begin{align*}
		\eta_{(F,G)}^{(1)} = \max\limits_{1 \leq i \leq N} \frac{|\langle g_i,f_i\rangle| + \|g_i\|\,\|f_i\|}{2} = \max\limits_{1 \leq i \leq N} \frac{|\langle Ug_i, Uf_i\rangle| + \|Ug_i\|\,\|Uf_i\|}{2} = \eta_{(UF,UG)}^{(1)}
	\end{align*}
Hence, $(F,G) \in \mathcal{N}^{(1)}$ if and only if $(UF,UG) \in  \mathcal{N}^{(1)}.$
\end{proof}

\section{Relations between the optimal dual pairs}
In this section, we analyze the relations between the three types of optimal dual pairs, which have been discussed in this paper. In view of Corollary \ref{thm5point4} and Theorems \ref{prop3point2},  \ref{thm4point3} and \ref{thm5point1}, we obtain the following implications.
\begin{thm}\label{T:6.1}
Let $(F,G)$ be an $(N, n)$ dual pair.
\begin{enumerate}
\item [{\em (i)}] If $(F,G) \in \mathcal{F}^{(1)},$ then  $(F,G) \in \mathcal{R}^{(1)}.$
\item [{\em (ii)}] If $(F,G) \in \mathcal{F}^{(1)}$, then $(F,G) \in \mathcal{N}^{(1)}$ and vice versa.
\item [{\em (iii)}] If $(F,G) \in \mathcal{N}^{(1)}$, then $(F,G) \in \mathcal{R}^{(1)}.$
\end{enumerate}
\end{thm}
The converse of the first and last statements in the above theorem need not always be true, as indicated by Example \ref{Eg:5.2}. However, they are true if $F$ is a tight frame and $G=S_{F}^{-1}F$, which is stated below.

\begin{thm}
Let $F$ be a tight frame for $\mathcal{H}.$ Then the following are equivalent.
\begin{enumerate}
\item [{\em (i)}] $(F, S_{F}^{-1}F)$ is a $1-$erasure optimal dual pair under Frobenius norm.
\item [{\em (ii)}]  $(F, S_{F}^{-1}F)$ is a $1-$erasure spectrally optimal dual pair.
\item [{\em (iii)}]  $(F, S_{F}^{-1}F)$ is a $1-$erasure numerically optimal dual pair.
\end{enumerate}
\end{thm}

\begin{proof} Let $F$ be a tight frame with the optimal frame bound $A.$ In view of Theorem \ref{T:6.1}, all that we need to show is the equivalence between $(F,S_{F}^{-1}F) \in \mathcal{F}^{(1)}$ and $(F,S_{F}^{-1}F) \in \mathcal{R}^{(1)}.$ This is an immediate consequence of Theorems \ref{prop3point2}, \ref{thm4point3} and the following relation:
$$\hspace*{1.2in}\|f_i\| \,\left\| S_{F}^{-1}f_i\right\|=\frac{1}{A}\|f_i\|^2=\langle f_i,  S_{F}^{-1}f_i\rangle,\text{ for every } i.\hspace*{1.2in}\qedhere$$
\end{proof}

\subsection*{Acknowledgment} The first author S. Arati acknowledges the financial support of National Board for Higher Mathematics, Department of Atomic Energy, Government of India.

\end{document}